\documentclass{amsart}

\usepackage{amssymb}
\usepackage[all]{xy}
\usepackage{hyperref}
\newtheorem{thm}{Theorem}%[section]

\newtheorem{lem}[thm]{Lemma}
\newtheorem{cor}[thm]{Corollary}

\theoremstyle{definition}

\newtheorem{say}[thm]{}

\newtheorem{rem}[thm]{Remark}          

\newtheorem*{ack}{Acknowledgments}      % \renewcommand{\theack}{} 

\newtheorem{defn-thm}[thm]{Definition--Theorem}  %!!!!!!!!!!!!!!!!!!!!!!!!
\newtheorem{defn-lem}[thm]{Definition--Lemma}  %!!!!!!!!!!!!!!!!!!!!!!!!
\theoremstyle{remark}

\renewcommand{\o}[0]{{\mathcal O}} 
\newcommand{\q}[0]{{\mathbb Q}}

\newcommand{\proj}[0]{\operatorname{Proj}}

\newcommand{\aut}[0]{\operatorname{Aut}}

\newcommand{\ex}[0]{\operatorname{Ex}}

\newcommand{\res}[0]{\operatorname{\mathcal R}} 
\newcommand{\alt}[0]{\operatorname{\mathcal A}}

\newcommand{\tsum}[0]{\textstyle{\sum}}

\newcommand{\shom}[0]{\operatorname{\mathcal{H}\!\it{om}}}

\def\loccoh#1.#2.#3.#4.{H^{#1}_{#2}(#3,#4)}

\DeclareMathAlphabet{\mathchanc}{OT1}{pzc}%
                                {m}{it}

\usepackage[all]{xy}\xyoption{dvips}

\begin{document}
\bibliographystyle{amsalpha}

%\hfill\today

 \title[Resolution and alteration]{Resolution and alteration with \\ ample exceptional divisor}
 \begin{abstract}In this short note we explain how to construct resolutions or regular alterations admitting an ample exceptional divisor, assuming the existence of projective resolutions or regular alterations. In particular, this implies the existence of such resolutions for arithmetic three-dimensional singularities.
\end{abstract}
\author{J\'anos Koll\'ar and Jakub Witaszek}

 \maketitle

It is frequently advantageous to have resolutions or alterations that have an 
ample exceptional divisor. While Hironaka-type methods automatically produce such a resolution, neither the resolution of 3-dimensional schemes
\cite{cos-pil-2014}  nor  alterations \cite{deJ-alt} yield ample exceptional divisors right away.
The aim of this note is to outline a simple trick that does ensure
the existence of ample exceptional divisors. 

Let $X$ be  an integral scheme.
A proper, birational morphism
 $\pi \colon Y \to X$ is a \emph{resolution}  if $Y$ is regular, and a 
 \emph{log resolution} if, in addition,   the exceptional locus  $\ex(\pi)$ is a simple normal crossing divisor.  
A proper, dominant, generically finite  morphism $\pi \colon Y \to X$ is an
\emph{alteration.} It is called
\emph{regular} if  $Y$ is regular, and \emph{Galois} with group $G=\aut(Y/X)$ if $Y/G\to X$ is generically purely inseparable. 
We let $\ex(\pi)\subset Y$ denote the smallest
closed subset such that $\pi$ is quasi-finite on $Y\setminus\ex(\pi)$.

\begin{thm} \label{jw.thm.1} Let $X$ be a  Noetherian, normal scheme. Assume that projective
resolutions (resp.\ log resolutions) exist for every  
scheme  $X'\to X$ that is projective and birational over $X$.

Then  $X$  has a  projective resolution (resp.\ log resolution)
$g:\res(X)\to X$ by a scheme $\res(X)$, such that $\ex(g)$ supports a $g$-ample divisor.
\end{thm}

\begin{thm} \label{jw.thm.2}
Let $X$ be a  Noetherian, normal scheme. Assume that  regular, projective, 
Galois  alterations  exist for every  
scheme  $X'\to X$ that is projective and generically purely inseparable  over $X$.

Then  $X$  has a  regular, projective, Galois alteration 
$g:\alt(X)\to X$  by a scheme $\alt(X)$, such that $\ex(g)$ supports a $g$-ample divisor.
\end{thm}

Note that  Theorems \ref{jw.thm.1}--\ref{jw.thm.2} are also valid for algebraic spaces and  stacks; see Remark~\ref{rem:algebraic-spaces-and-stacks} for details.

\begin{cor} \label{cor.1}
Let $X$ be a normal, integral, quasi-excellent scheme (or algebraic space) of dimension at most three, that is separated and of finite type over an affine quasi-excellent scheme $S$. Then $X$ admits a  projective log resolution $g:\res(X)\to X$ by a scheme $\res(X)$, such that $\ex(g)$ supports a $g$-ample divisor.
\end{cor}

\begin{cor}\label{cor.2}
Let $X$ be a Noetherian, normal, integral scheme (or algebraic space), that is separated and of finite type over an excellent scheme $S$ with $\dim S \leq 2$. Then $X$ admits a  regular, projective, Galois alteration
$g:\alt(X)\to X$  by a scheme $\alt(X)$, such that $\ex(g)$ supports a $g$-ample divisor.
\end{cor}

\begin{rem} It is clear from the proof that one can find 
$g:\res(X)\to X$ and $g:\alt(X)\to X$ with other useful properties. For example, we can choose $\res(X)$ (resp.\ $\alt(X)$) to dominate any finite number of
resolutions  (resp.\ alterations). 

Also, if $Z_i\subset X$ are finitely many closed subschemes, and
embedded resolutions  (resp.\  regular, Galois alterations) exist over $X$,
then we can choose $\res(X)$ (resp.\ $\alt(X)$) to be an 
embedded resolution  (resp.\ regular, Galois  alteration)  for the $Z_i$. 

The log version of alterations does not seem to be treated in the literature.
\end{rem}

To fix our notation, recall that a normal scheme $X$ is $\q$-factorial if, for every  generically invertible sheaf $L$, there is an $m>0$ such that $L^{[m]}$
(the reflexive hull of $L^{\otimes m}$) is invertible.

We start with three lemmas; the first two are well known. 

\begin{lem}\label{jw.lem.0}
 Let $X$ be a  Noetherian, normal, $\q$-factorial  scheme,    $\pi:X'\to X$ a projective, birational morphism with $X'$ normal.  
Then there is a $\pi$-ample,  $\pi$-exceptional divisor  $E$ on $X'$.
\end{lem}
\begin{proof} Let   $H$  be  a $\pi$-ample  line bundle on $X'$. 
Choose $m>0$ such that $(\pi_*H)^{[m]}$ is invertible.
Then $H^m\otimes \pi^*\bigl((\pi_*H)^{[-m]}\bigr)$ is $\pi$-ample
and trivial on $X'\setminus \ex(\pi)$.  Thus it is linearly equivalent to a 
$\pi$-exceptional divisor  $E$.
\end{proof}

\begin{lem}\label{jw.lem.1}
 Let $X$ be a  Noetherian, normal  scheme,    $\pi_1:X_1\to X$ a projective, generically purely inseparable morphism, and 
  $H_1$   a  line bundle on $X_1$. 
Set $U_1:=X_1\setminus \ex(\pi_1)$. 

Then there is a coherent, generically invertible sheaf $L_1$ on $X$  and $q>0$, such that, 
$\pi_1^* L_1|_{U_1}\cong H_1^q|_{U_1}$.
\end{lem}
\begin{proof}
Consider the  Stein factorization
$X_1\xrightarrow{\rho'} X' \xrightarrow{\rho} X$ of $\pi$.   
The images of $U_1$ give  $U'\subset X'$ and $U\subset X$. So
$\rho'_*H_1$ is a line bundle on $U'$. Since $U'\to U$ is finite and purely inseparable,  it factors through a power of Frobenius; cf.\ \cite[Tag 0CNF]{stacks-project}. Hence there is a line bundle $L_U$ on $U$ such that
$\rho^*L_U\cong \rho'_*H_1^q|_{U'}$, where we can take $q=\deg \rho$.
We can then extend $L_U$ to a coherent sheaf $L_1$ on $X$. 
\end{proof}

\begin{lem}\label{jw.lem.2}
 Let $X$ be a  Noetherian, normal  scheme and    $\pi_1:X_1\to X$ a projective, generically purely inseparable morphism.
  Assume that $X_1$ is $\q$-factorial and let 
  $H_1$  be a $\pi_1$-ample line bundle on $X_1$.
Let $L_1$ be a coherent, generically invertible sheaf  on $X$ as in Lemma \ref{jw.lem.1}.
Set   $L_2:=\shom_X(L_1, \o_X)$ and 
$$
\pi_2: X_2:=\proj_X \tsum_{m\geq 0}L_2^{\otimes m}\to X.
$$
Let  $\pi_3:X_3\to X$ be a projective, generically purely inseparable morphism
that dominates both $X_1$ and $X_2$.  
Then there is a $\pi_3$-ample,  $\pi_3$-exceptional divisor  $E$ on $X_3$.
\end{lem}
\begin{proof}
Let $\tau_i:X_3\to X_i$ be the natural maps,  $H_2:=\o_{X_2}(1)$, and  
$X_3\xrightarrow{\tau'} X'_1 \xrightarrow{\tau} X_1$  the Stein factorization of 
$\tau_1$. Since $X_1$ is $\q$-factorial and $\tau$ is finite and purely inseparable (and so, as above, it is an isomorphism or it factors through a power of Frobenius), $X'_1$ is also $\q$-factorial. 

By Lemma~\ref{jw.lem.0}  there is a
$\tau'$-ample,  $\tau'$-exceptional divisor  $E_3$ on $X_3$. 
Then  $\tau_1^*H_1^m(E_3)$ is   $\pi_3$-ample for  $m\gg 0$. 

Since  $H_2$ is $\pi_2$-nef, its pull-back $\tau_2^*H_2$ is  $\pi_3$-nef.
Therefore 
$\tau_2^*H_2^m\otimes \tau_1^*H_1^{qm}(E_3)$ is   $\pi_3$-ample as well, where $q$ is as in Lemma~\ref{jw.lem.1}. 

Set  $U_3:=X_3\setminus \ex(\pi_3)$; its images  give open subschemes 
 $U\subset X$ and $U_i\subset   X_i$.  
Then
$$
\tau_2^*H_2^m\otimes \tau_1^*H_1^{qm}(E_3)|_{U_3}\cong 
\pi_3^*\bigl(L_2^m|_U\otimes L_1^m|_U\bigr)\cong \o_{U_3}.
$$
This gives a rational section of $\tau_2^*H_2^m\otimes \tau_1^*H_1^{qm}(E_3)$
whose divisor is $\pi_3$-ample and  $\pi_3$-exceptional.
\end{proof}

\begin{say}[Proof of Theorem~\ref{jw.thm.1}] 
Start with a projective (log) resolution  $\pi_1:X_1\to X$ and construct
$\pi_2:X_2\to X$ as in Lemma  \ref{jw.lem.2}. Let  $X_{12}\subset X_1\times_XX_2$ 
be the irreducible component that dominates $X$, and  $X_3\to X_{12}$  a projective (log) resolution.
By Lemma \ref{jw.lem.2}, $\pi_3:X_3\to X$  has a $\pi_3$-ample, $\pi_3$-axceptional   divisor.
\qed
\end{say}

\begin{say}[Proof of Theorem~\ref{jw.thm.2}] 
Start with a regular, projective, Galois  alteration  $\bar\pi_1:\bar X_1\to X$.
Let $\pi_1:X_1\to X$ be its quotient by the Galois group of $k(\bar X_1/X)$.
Note that $X_1$ is $\q$-factorial.

Construct
$\pi_2:X_2\to X$ as in Lemma \ref{jw.lem.2}. Let  $X_{12}\subset X_1\times_XX_2$ 
be the irreducible component that dominates  $X$, and $\bar X_3\to X_{12}$ a regular , projective, Galois alteration.  Let
$X_3\to X_{12}$ be its quotient by the Galois group of $k(\bar X_3/X_{12})$.
By Lemma \ref{jw.lem.2}, $\pi_3:X_3\to X$  has a $\pi_3$-ample, $\pi_3$-axceptional   divisor.
Its pull-back to $\bar X_3$ is  a $\bar \pi_3$-ample, $\bar \pi_3$-exceptional   divisor, where $\bar \pi_3 \colon \bar X_3 \to X$ is the natural morphism.
\qed
\end{say}
\begin{rem} \label{rem:algebraic-spaces-and-stacks}
Theorems \ref{jw.thm.1}--\ref{jw.thm.2}   are valid for every integral, Noetherian algebraic space (resp.\ stack) $X$ with $\res(X)$ or $\alt(X)$ being an algebraic space (resp.\  stack), assuming the appropriate representable resolutions or  regular alterations by algebraic spaces (resp.\  stacks) exist for every algebraic space (resp.\  stack) $X'$ admitting a representable projective birational  (resp.\ generically purely inseparable) morphism to $X$. As for algebraic spaces, we note that all of the above constructions can be performed in the category of algebraic spaces and their validity may be verified \'etale locally. As for algebraic stacks, we note that every algebraic stack admits a presentation as a quotient of an algebraic space by a smooth groupoid \cite[Tag 04T3]{stacks-project}, and that quotients of algebraic spaces by smooth groupoids always exist \cite[Tag 04TK]{stacks-project}. We can then conclude as each step in our constructions is equivariant with respect to a chosen presentation.

If $X$ is an algebraic space and the appropriate resolutions or  regular alterations of all algebraic spaces admiting representable, projective, birational or generically purely inseparable morphisms to $X$ exist as schemes, then we can assume that $\res(X)$ or $\alt(X)$ is a scheme. 

Here, a representable morphism of quasi-compact quasi-separated algebraic spaces (resp.\ algebraic stacks) is \emph{projective} if it is proper and there exists a relatively ample invertible sheaf (cf.\ \cite[Definition 8.5 and Theorem 8.6]{Rydh-stacks15}).
\end{rem}

\begin{say}[Proof of Corollary~\ref{cor.1}]
When $X$ is a scheme, the assumptions of Theorem \ref{jw.thm.1} are valid for integral affine quasi-excellent schemes of dimension at most three  by \cite{cos-pil-2014}, see \cite[Theorem 2.5 and 2.7]{many-p}. 

If $X$ is an algebraic space, then by Chow's lemma \cite[Tag 088U]{stacks-project} we can find a projective birational morphism $h \colon Y \to X$ such that the scheme $Y$ is quasi-projective over $S$. 
 M.\ Temkin extended \cite{cos-pil-2014} to give a projective resolution for such a scheme $Y$; the proof  will be contained in the revised version  of \cite{many-p}.

Similarly, we obtain projective resolutions of all algebraic spaces admitting a projective birational morphism to $X$. By Remark~\ref{rem:algebraic-spaces-and-stacks} we can  obtain $\res(X)$ as a scheme.
\end{say}

\begin{say}[Proof of Corollary~\ref{cor.2}]
When $X$ is a scheme, the assumptions of Theorem \ref{jw.thm.2} are valid for all integral schemes that are separated and of finite type over an excellent scheme $S$ with $\dim S \leq 2$ (see \cite[Corollary 5.15]{MR1450427} and
\cite[4.3.1]{tem-dist}).

If $X$ is an algebraic space, then a regular, projective,  Galois alteration of $X$ (and of all algebraic spaces admitting a projective generically purely inseparable morphism to $X$) exists by Chow's lemma as in the proof of Corollary~\ref{cor.1}, and so we can conclude by Remark \ref{rem:algebraic-spaces-and-stacks} to get $\alt(X)$,  which is a scheme.
\end{say}

\begin{rem}
The above proofs of  Corollaries~\ref{cor.1}--\ref{cor.2} do not immediately apply to  algebraic stacks. Indeed, Chow's lemma for algebraic stacks only ensures the existence of a proper surjective cover by a quasi-projective scheme. This cover need not be birational. On the other hand, one could try to construct a resolution equivariantly with respect to a presentation, but we do not know whether the algorithms for the existence of resolutions and regular alterations from \cite{cos-pil-2014} and \cite{deJ-alt} can be run equivariantly (in contrast to the characteristic zero case). For Deligne-Mumford stacks of finite type over a Noetherian scheme, the proper surjective cover from Chow's lemma may be assumed to be generically \'etale \cite[Corollaire 16.6.1]{L-MB2000}. In particular, they admit regular  alterations (and so also regular, Galois  alterations) and Corollary~\ref{cor.2} holds for them.   
\end{rem}

\begin{ack} We thank M.~Temkin for very helpful e-mails on alterations and B.~Bhatt, L.~Ma, Z.~Patakfalvi, K.~Schwede, K.~Tucker, and J.~Waldron for valuable conversations.
Partial  financial support  to JK   was provided  by  the NSF under grant number
DMS-1901855.
\end{ack}

%\bibliography{refs}

\newcommand{\etalchar}[1]{$^{#1}$}
\def\cprime{$'$} \def\cprime{$'$} \def\cprime{$'$} \def\cprime{$'$}
  \def\cprime{$'$} \def\dbar{\leavevmode\hbox to 0pt{\hskip.2ex
  \accent"16\hss}d} \def\cprime{$'$} \def\cprime{$'$}
  \def\polhk#1{\setbox0=\hbox{#1}{\ooalign{\hidewidth
  \lower1.5ex\hbox{`}\hidewidth\crcr\unhbox0}}} \def\cprime{$'$}
  \def\cprime{$'$} \def\cprime{$'$} \def\cprime{$'$}
  \def\polhk#1{\setbox0=\hbox{#1}{\ooalign{\hidewidth
  \lower1.5ex\hbox{`}\hidewidth\crcr\unhbox0}}} \def\cdprime{$''$}
  \def\cprime{$'$} \def\cprime{$'$} \def\cprime{$'$} \def\cprime{$'$}
\providecommand{\bysame}{\leavevmode\hbox to3em{\hrulefill}\thinspace}
\providecommand{\MR}{\relax\ifhmode\unskip\space\fi MR }
\providecommand{\MRhref}[2]{%
  \href{http://www.ams.org/mathscinet-getitem?mr=#1}{#2}
}
\providecommand{\href}[2]{#2}

\bigskip

   Princeton University, Princeton NJ 08544-1000, \

 \email{kollar@math.princeton.edu}

\bigskip 
 
 University of Michigan, Ann Arbor MI 48109 \

\email{jakubw@umich.edu}

\end{document}